\documentclass[12pt,a4paper]{article}

\usepackage{amsmath}
\usepackage{amsthm}
\usepackage{amssymb}
\usepackage[T1]{fontenc}
\usepackage{lmodern}
\usepackage{exscale}
\usepackage{calrsfs}
\usepackage{graphicx}
\date{}

\numberwithin{equation}{section}
\numberwithin{figure}{section}

\newtheorem{thm}{Theorem}
  
\newtheorem{thrm}{Theorem}[section]
\newtheorem{prop}[thrm]{Proposition}

\newtheorem{cor}[thrm]{Corollary}
\newtheorem{coro}[thm]{Corollary}
\theoremstyle{remark}

\newtheorem*{rmk}{Remark}

\theoremstyle{definition}

\newtheorem*{expl}{Examples}

\renewcommand{\epsilon}{\varepsilon}
\newcommand{\R}{\mathbb{R}}
\newcommand{\lft}[1]{\left#1}
\newcommand{\rgt}[1]{\right#1}

\newcommand{\rd}{\partial}
\newcommand{\uv}[1]{\frac{\partial}{\partial #1}}
\newcommand{\uvv}[1]{({\partial}/{\partial #1})}

\newcommand{\sect}[1]{\operatorname{Sec}\left(#1\right)}
\newcommand{\secc}[1]{\operatorname{Sec}#1}
\newcommand{\D}{\mathcal{D}}
\newcommand{\DD}{\underbar{$\D$}}
\newcommand{\ub}[1]{\underbar{$#1$}}
\newcommand{\rjet}[1]{{#1}^{(r)}}
\newcommand{\ojet}[1]{{#1}^{(1)}}
\newcommand{\vertt}{\operatorname{Vert}}
\newcommand{\symb}[1]{\operatorname{Symb}#1}
\newcommand{\dsec}[1]{\operatorname{Sec}_{\mathcal{F}}(#1)}

\begin{document} 
%
%
\title{Existence and classification of maximally non-integrable distributions 
of derived length one}
\author{Jiro ADACHI}

\maketitle

\renewcommand{\thefootnote}{\fnsymbol{footnote}}%
\footnote[0]{This work was supported 
  by JSPS KAKENHI Grant Number 25400077.}%

\begin{abstract}
  The \textit{h}-principles for maximally non-integrable tangent distributions 
of derived length one on manifolds are studied in this paper. 
 Such distributions of odd rank are dealt with. 
 The formal structures for such distributions are introduced. 
 From the view point of the \textit{h}-principles, 
we discuss the existence and classification of such structures. 
\end{abstract} 

\section{Introduction}\label{sec:intro}
Existence and Classification of geometric structures on manifolds 
are fundamental problems in differential topology. 
 In this paper, we discuss such problems 
for maximally non-integrable tangent distributions of derived length one 
of odd rank. 
 We obtain the necessary and sufficient condition 
for the existence of such distributions
on manifolds possibly closed. 
 In addition, we discuss classification of such structures up to isotopy. 

  First, we introduce the structures that we deal with in this paper. 
 Let $V$ be a manifold of dimension~$n$. 
 A \emph{tangent distribution}\/ on $V$ of rank~$r$ 
is a subbundle $\mathcal{D}\subset TV$ of the tangent bundle of $V$ 
whose fiber is of dimension~$r$. 
 At each point $x\in V$, a subspace 
$\mathcal{D}^2_x=[\mathcal{D},\mathcal{D}]_x+\mathcal{D}_x\subset T_xV$ 
is derived from $\mathcal{D}$ by the Lie bracket 
(see Section~\ref{sec:distr} for precise definition). 
 If $\mathcal{D}^2_x=\mathcal{D}_x$ at any $x\in V$, 
$\mathcal{D}$ is said to be \emph{integrable}. 
 When a distribution is integrable, it generates a foliation on $V$. 
 We deal with non-integrable distributions on this paper. 
 We impose further conditions on tangent distributions. 
 A tangent distribution $\mathcal{D}\subset TV$ 
is said to be \emph{completely non-holonomic\/} with derived length~$1$ 
if $\mathcal{D}^2=TV$. 
 We deal with such tangent distributions of odd rank $r=2k+1$. 
 A distribution $\mathcal{D}\subset TV$ on $V$ of rank~$r$ 
is locally described as a kernel of an $(n-r)$-tuple of $1$-forms 
$\alpha_1,\alpha_2,\dots,\alpha_{n-r}$: 
\begin{equation*}
  \mathcal{D}=\{\alpha_1=0,\ \alpha_2=0,\dots,\alpha_{n-r}=0\}. 
\end{equation*}
 We say a distribution $\mathcal{D}$ is \emph{maximally non-integrable\/} 
if $(d\alpha_1)^k,\dots,(d\alpha_{n-r})^k$ 
are linearly independent $(2k)$-forms at each point 
(see Section~\ref{sec:distr} for precise definition). 
 We remark that, from this condition, 
$\mathcal{D}$ is of derived length one (see Section~\ref{sec:distr}). 
 In this paper we deal with such tangent distributions, 
maximally non-integrable tangent distributions 
of derived length one and of odd rank. 
 Note that, if the corank $n-r=n-(2k+1)$ of such a tangent distribution is $1$, 
it is the so-called even-contact structure or quasi-contact structure. 
 In that sense, maximally non-integrable distribution of derived length one 
is regarded as a generalization of even-contact structure. 

  We introduce, in this paper, the formal structure 
for maximally non-integrable tangent distribution of derived length one. 
 Let $V$ be a manifold of dimension~$n$, 
and $\mathcal{D}\subset TV$ a tangent distribution of rank~$r=2k+1$. 
 We say the distribution $\mathcal{D}$ 
is \emph{almost maximally non-integrable of derived length one}\/ 
if there exist an $(n-r)$-tuple of $2$-forms 
$\omega_1,\ \omega_2,\dots,\omega_{n-r}$ 
for that $(\omega_1)^k,(\omega_2)^k,\dots,(\omega_{n-r})^k$ 
are linearly independent on $\mathcal{D}$ at any point. 
 Clearly, maximally non-integrable distribution of derived length one 
is almost maximally non-integrable distribution of derived length one 
for $\omega_i=d\alpha_i$. 

  Now, the existence theorem for maximally non-integrable distributions 
of derived length one is described as follows. 
%
%
\begin{thm}\label{thma}
  Let $V$ be a manifold of dimension~$n$ possibly closed. 
 The manifold $V$ admits a maximally non-integrable distribution 
of derived length one and of rank~$2k+1$ 
if and only if
it admits an almost maximally non-integrable distribution 
of derived length one and of rank~$2k+1$. 
\end{thm} 

  Recall that, if the corank $n-r=n-(2k+1)$ of the distribution 
$\mathcal{D}\subset TV$ is $1$, 
the tangent distribution $\mathcal{D}$ is an even-contact structure. 
 We should remark that, in this case, 
this theorem is equivalent to the theorem on McDuff's \textit{h}-principle 
for even-contact structures (see~\cite{mcduff}). 
 In this sense, Theorem~\ref{thma} in this paper 
can be regarded as a generalization of McDuff's \emph{h}-principle. 

  The first explicit generalization may be a distribution of type~$(3,5)$. 
 It is, by definition, a tangent distribution $\mathcal{D}$ of rank~$3$ 
on a manifold $V$ of dimension~$5$ that satisfies $\mathcal{D}^2=TV$. 
 When we deal with orientable distributions, we obtain the following 
as a corollary of Theorem~\ref{thma}. 
%
%
\begin{coro}\label{corA}
  Let $V$ be a manifold of dimension~$5$ possibly closed. 
  The manifold $V$ admits an orientable distribution of type~$(3,5)$ 
  if and only if it admits a trivial subbundle $\mathcal{E}\subset TV$ 
  of rank~$3$ of the tangent bundle. 
\end{coro} 
\noindent
 In Section~\ref{sec:proof}, 
we discuss a more general corollary of Theorem~\ref{thma} 
for orientable distributions than Corollary~\ref{corA}. 

  Then the result concerning the classification is described as follows. 
 It is also considered as a generalization 
of McDuff's \textit{h}-principle (see~\cite{mcduff}). 
%
%
\begin{thm}\label{thmb}
  Let $\mathcal{D}_1,\ \mathcal{D}_2\subset TV$ 
be maximally non-integrable distributions of derived length one 
and of odd rank on a manifold $V$. 
 If they are homotopic through almost maximally non-integrable distributions 
then they are isotopic. 
\end{thm} 
\noindent
 This implies that the isotopic classification 
of maximally non-integrable distributions of derived length one 
follows a ``topological'' classification. 
 In other words, all such structures are ``flexible'', 
and there is no ``rigid'' class. 
 The explicit classification is given by algebraic topology 
on each base manifold $V$. 

  These theorems are proved from the view point of the \textit{h}-principles. 
Gromov's convex integration method is applied to show the \emph{h}-principles. 
 As we mentioned above, these results are generalizations of 
McDuff's \textit{h}-principle for even-contact structures. 
 It is originally proved by using Gromov's convex integration method 
(see~\cite{mcduff}), 
although an alternative prof is given in~\cite{elmi_book}. 
 The proof in this paper is influenced by the original one. 

  Let us observe some related results around. 
 A completely non-integrable distribution of derived length one 
is a \emph{contact structure}\/ if it is of corank~$1$ and even rank. 
 The existence of a contact structure on a closed orientable $3$-manifold 
is proved by Martinet~\cite{martinet71} by a constructive method. 
 For higher-dimensional manifolds, 
it is proved by Borman, Eliashberg, and Murphy~\cite{boelmu} 
from the view point of the \textit{h}-principle 
that, if a manifold admits an almost contact structure, 
it admits a genuine contact structure. 
 One of the properties that make contact topology important 
is the existence of the rigid (tight) class. 
 It is proved by Eliashberg~\cite{eliash92} 
that contact structures on $3$-manifolds 
are divided into two classes, tight (rigid) and overtwisted (flexible). 
 For higher-dimensional manifolds, 
the overtwistedness is introduced in~\cite{boelmu}. 
 On the other hand, McDuff's \textit{h}-principle~\cite{mcduff} implies 
that there is no rigid class for even-contact structures. 
 In this sense, the result in this paper implies 
that there is no rigid class 
for maximally non-integrable distributions of derived length one of odd rank. 

  This paper is organized as follows. 
 In the following section, we introduce some basic notions for this paper. 
 First, we review something on tangent distributions, 
and then we introduce the \emph{h}-principles and convex integration method. 
 In Section~\ref{sec:diff-rel}, we set the problem to be solved 
in terms of the \textit{h}-principles. 
 Then we show the \textit{h}-principles by Gromov's convex integration method 
in Section~\ref{sec:pf_ethm}. 
 The proof is concluded in Section~\ref{sec:proof}. 

\smallskip

  The author is grateful to Goo Ishikawa for fruitful discussions.

\section{preliminaries}\label{sec:prelim}
  In this section, we discuss basic things 
which are needed in the proof of Theorems. 
 In Subsection~\ref{sec:distr}, we introduce tangent distributions on manifolds.
 In Subsection~\ref{sec:h-prin}, the notion of the \textit{h}-principles 
and the method of convex integration are introduced. 

\subsection{Tangent distribution}\label{sec:distr}
  First of all, we introduce the notion of tangent distributions on manifolds 
and something concerning the notion. 
 Let $M$ be a smooth manifold of dimension $n$. 
 A subbundle $\D\subset TM$ of the tangent bundle $TM$ 
with $r$-dimensional fibers 
is called a \emph{tangent distribution} (or \emph{distribution} for short) 
of rank $r$ on $M$. 
 Let $\sect{\D}$ be the set of all cross-sections (vector fields) 
of the subbundle $\D\to M$, and $\DD\subset \ub{TM}$ 
the sheaf of the vector fields. 
 Then $\D$ can be considered as a distribution 
of the $r$-dimensional tangent subspace $\D_p\subset T_pM$ 
at each point $p\in M$, 
or a $r$-dimensional plane field on $M$. 

  In this paper, we deal with completely non-holonomic distributions, 
which are defined as follows. 
 Let $\D$ be a distribution of rank~$r$ on an $n$-dimensional manifold $M$. 
 A distribution $\D\subset TM$ is said to be \emph{completely non-holonomic} 
if, for any local frame $\{X_1,\dots,X_k\}$ of $\D$, 
its iterated Lie brackets $X_i,[X_i,X_j],[X_i,[X_j,X_l]],\dots$ 
spans the tangent bundle $TM$. 
 This condition is observed from the view point of distributions 
derived by the Lie brackets as follows. 
 For a given local vector field $X\in \D$ defined 
on a neighborhood of a point $x\in M$, 
we have the germ $\ub{X}_x\in \DD(x)$ 
as an element of the stalk $\DD(x)$ at $x\in M$. 
 Then by inductively setting 
\begin{align*}
  \DD(x)^2&=([\DD,\DD]+\DD)_x \\
          &:=\operatorname{Span}\{[\ub{X}_x,\ub{Y}_x]+\ub{Z}_x\in\ub{TM}(x)
            \mid \ub{X}_x,\ub{Y}_x,\ub{Z}_x\in\DD(x)\}, \\
  \DD(x)^{l+1}&=([\DD,\DD^l]+\DD^l)_x \\ 
          &:=\operatorname{Span}\{[\ub{X}_x,\ub{Y}_x]+\ub{Z}_x\in\ub{TM}(x)
  \mid \ub{X}_x\in\DD(x),\ub{Y}_x,\ub{Z}_x\in\DD(x)^l\},\quad (l=2,3,\dots), 
\end{align*}
we have a flag 
$\DD\subset\DD^2\subset\dots\subset\DD^l\subset\dots\subset\ub{TM}$ 
of subsheaves. 
 The condition that the distribution $\D\subset TM$ is completely non-holonomic 
implies that $\DD^l=\ub{TM}$ holds for some $l\in\mathbb{N}$. 
 We have the subspace $\D^l_x\subset T_{x}M$ as 
$
  \D^l_x:=\{X_x\in T_{x}M\mid \ub{X}_x\in \DD(x)^l\} 
$. 
 If $\dim(\D^l_x)$ is constant for any $x\in M$, 
there corresponds a distribution $\D^l\subset TM$ on $M$. 
 It is called the \emph{derived distribution}\/ of $\D$. 
 In this paper, we deal with completely non-holonomic distributions 
of \emph{derived length one}. 
 In other words, completely non-holonomic distributions that satisfy $\D^2=TM$. 
 When the rank of such a tangent distribution $\mathcal{D}$ on $M$ is $r$ 
and the dimension of the manifold $M$ is $n$, 
then $\mathcal{D}$ is said to be \emph{of type~$(r,n)$}. 

  Now, we introduce basic property 
concerning completely non-holonomic distributions of derived length one. 
 Let $\D\subset TM$ be a distribution of rank~$r$ 
on an $n$-dimensional manifold $M$. 
 Let $\mathcal{S}(\D)\subset T^\ast M$ be the bundle of covectors 
that annihilate $\D$. 
 Note that $\mathcal{S}(\D)$ is a vector bundle of fiber dimension~$n-r$. 
 The distribution $\D$ can be locally described by $1$-forms 
$\omega_1,\omega_2,\dots,\omega_{n-r}\in\sect{T^\ast M}$ as 
$\D=\{\omega_1=0,\ \omega_2=0,\dots,\ \omega_{n-r}=0\}$. 
 We observe the choice of the basis 
$\mathcal{S}(\D)$ annihilating the distribution $\D$. 
 If $\D$ is completely non-holonomic with derived length one, 
then we can take a basis $\{\omega_1,\ldots,\omega_{n-r}\}$ as follows 
(see~\cite{art11} for example). 
%
%
\begin{prop}\label{prop:dbasis}
  Let $\D\subset TM$ be a distribution of rank~$r$ 
on an $n$-dimensional  manifold $M$. 
 The distribution $\D$ is completely non-holonomic with derived length one 
if and only if $\mathcal{S}(\D)$ has a local basis 
$\{\omega_1,\ \omega_2,\dots,\ \omega_{n-r}\}$ 
which satisfies the condition that $(n-r+2)$-forms
$\omega_1\wedge\omega_2\wedge\dots\wedge\omega_{n-r}\wedge d\omega_i$, 
$i=1,2,\dots,n-r$, 
are vanishing nowhere and  pointwise linearly independent. 
\end{prop} 

  In addition, we impose further conditions to distributions 
that we are dealing with in this paper. 
 Let $\D\subset TM$ be a distribution of rank~$r$ 
on an $n$-dimensional manifold $M$. 
 Let $\{\omega_1,\ \omega_2,\dots,\omega_{n-r}\}$ 
be the defining $1$-forms of $\D$. 
 When $\D$ is of odd rank~$r=2k+1$, 
we say $\D$ is \emph{maximally non-integrable}\/ 
if it satisfies the following condition: 
\begin{itemize}
\item $\omega_1\wedge\omega_2\wedge\dots\wedge\omega_{n-2k-1}
  \wedge(d\omega_i)^{k}$,\ 
  $i=1,2,\dots,n-2k-1$,\ are pointwise linearly independent.  
\end{itemize}
 Note that such distributions are completely non-holonomic 
with derived length one from Proposition~\ref{prop:dbasis}. 
 When the rank of $\D$ is $r=3$, such notions are equivalent 
to completely non-holonomic with derived length one. 
 The dimension~$n$ of manifolds for such distributions of rank~$2k+1$ 
should be $2k+2\le n\le 2(2k+1)$. 

%
%
\begin{expl}
(1) A contact structure $\xi$ on a $(2n+1)$-dimensional manifold $M$ 
is completely non-holonomic with derived length one. 
 Actually, $\xi\subset TM$ is a distribution of corank~$1$ on $M$ 
which is completely non-integrable by definition. 
 By the Darboux theorem, it is locally contactomorphic 
to the standard contact structure $\xi_0=\{\alpha_0=0\}$, 
$\alpha_0=dz-\sum_{i=1}^ny_idx_i$ on $\R^{2n+1}$ 
with coordinates $(z,x_1,y_1,\dots,x_n,y_n)$. 
 This $\xi_0$ is spanned by $2n$ vector fields 
$\uvv{x_1}+y_1\uvv z, \uvv{y_1},\dots,\uvv{x_n}+y_n\uvv z, \uvv{y_n}$. 
 Since $[\uvv{x_i}+y_i\uvv z, \uvv{y_i}]=\uvv z$, we have $\xi^2=TM$. 
 We should remark that, from the same discussion, an even-contact structure is 
also completely non-holonomic with derived length one 
and maximally non-integrable.  \\
(2)\ The distribution $\D=\{dy-z_1dx_1=0\}$ on $\R^5$ 
with coordinates $(x_1,x_2,y,z_1,z_2)$ 
is completely non-holonomic with derived length one. 
 However, it is not completely non-integrable. \\
(3)\ The canonical distribution on the jet space $J^1(\R,\R^k)\cong\R^{2k+1}$ 
is completely non-holonomic with derived length one. 
 Let $(x,y_1,\dots,y_k,z_1,\dots,z_k)\in J^1(\R,\R^k)$ be the coordinates, 
where $z_i$ corresponds to $d y_i/d x$. 
 The canonical distribution $\D_0$ is given as 
\begin{equation*}
  \D_0:=\{\alpha_1=0,\dots,\alpha_k=0\},\quad 
  \text{where}\ \alpha_i:=dy_i-z_idx.
\end{equation*}
 It is a distribution of rank~$k+1$, or corank~$k$, 
spanned by the following vector fields: 
\begin{equation*}
  \uv x+\sum_{j=1}^kz_j\uv{y_j},\ \uv{z_1},\ \dots,\ \uv{z_k}. 
\end{equation*}
 Since $\lft[\uvv x+\sum_{j=1}^kz_j\uvv{y_j},\ \uvv{z_i}\rgt]=\uvv{y_i}$, 
$i=1,2,\dots,k$, we have $\D_0^2=T\R^{2k+1}$. 
 Note that, if $k=1$, it is the standard contact structure on $\R^3$. 
 If $k=2$, it is completely non-integrable. 
 However, for $k>2$ it is not completely non-integrable. 
\end{expl}

  In this paper we deal with maximally non-integrable distributions of odd rank.

\subsection{Homotopy principle and Convex integration}
\label{sec:h-prin}
  We introduce the notion of the \textit{h}-principles 
in Section~\ref{sec:diff_rel}. 
 It is a key tool to show Theorems~\ref{thma} and~\ref{thmb}. 
 In order to show that our objects satisfy the \textit{h}-principles, 
we use the convex integration method due to Gromov. 
 In Section~\ref{sec:cvx_itgr}, we introduce that method. 
 In order to make this paper self-contained, 
some basic things that are needed in this paper are defined there. 
 Readers should refer to the literature 
for further study on the \textit{h}-principles 
(see~\cite{gromov_pdr}, \cite{elmi_book}, \cite{madachi}, \cite{spring}). 

\subsubsection{Differential relations and homotopy principle}
\label{sec:diff_rel}
  We introduce a philosophy to ``solve'' some ``problems'' 
related to differentiation. 
  In other words, we introduce the notion of the \emph{h}-principle 
for partial differential relations. 

  First, we review basic notions concerning fibrations and jets. 
 Let $p\colon X\to V$ be a fibration over a manifold $V$. 
 The \emph{$r$-jet extension} $p^r\colon X^{(r)}\to V$ 
of the fibration $p\colon X\to V$ 
is defined as the manifold $X^{(r)}$ consisting of all $r$-jets 
of cross-sections $V\to X$ of the fibration at all $v\in V$ 
and the projection $p^r\colon X^{(r)}\to V$. 
 $X^{(r)}$ is called the \emph{$r$-jet space}. 
 The cross-section $J^r_f\colon V\to X^{(r)}$ of $p^r\colon X^{(r)}\to V$ 
is the $r$-jet extension of a cross-section $f\colon V\to X$. 

  Then, we introduce a description method for certain problems. 
 Let $p\colon X\to V$ be a fibration. 
 A \emph{differential relation} of order $r$ for sections 
of $p\colon X\to V$ 
is defined as a subset $\mathcal{R}$ of the $r$-jet space $\rjet X$. 
 For example, a system of differential equations is regarded 
as a subset of the jet space $J^r(\R^n,\R^m)$, 
that is, a differential relation. 
 There are two kinds of solutions to a differential relation. 
 A \emph{formal solution}\/ 
to the differential relation $\mathcal{R}\subset\rjet{X}$ 
is defined to be a section $F\colon V\to\rjet{X}$ 
of the fibration $p^r\colon \rjet{X}\to V$ 
that satisfies $F(V)\subset\mathcal{R}$. 
 Let $\secc\mathcal{R}$ denote the space of formal solution to $\mathcal{R}$. 
 On the other hand, a \emph{genuine solution}\/ to $\mathcal{R}\subset\rjet{X}$ 
is defined to be a section $f\colon V\to X$ of $p\colon X\to V$ 
whose $r$-jet satisfies $(J^r_f)(V)\subset\mathcal{R}$. 
 Let $\operatorname{Sol}\mathcal{R}$ denote the space of genuine solutions 
to $\mathcal{R}$. 

  Some examples of differential relations come from singularity theory. 
 Roughly speaking, in some cases, 
a singular point 
is a point 
where certain expressions related to derivatives vanish. 
 Then, to singular points, there corresponds the set $\Sigma\subset\rjet X$ 
called the \emph{singularity}. 
 Setting $\mathcal{R}:=\rjet X\setminus\Sigma$, 
we have an open differential relation. 
 In many cases, it is important for singularity theory 
to solve such differential relations. 

  Now, we introduce the notion of the homotopy principle. 
 In general, the formal solvability of a differential relation 
is just a necessary condition 
for the genuine solvability of the differential relation. 
 In some cases, the formal one is sufficient for the genuine solvability. 
 The notion, the homotopy principle, is introduced by Gromov and Eliashberg 
to formalize such properties (see~\cite{groeli71}). 
 Let $p\colon X\to V$ be a fibration, 
and $\mathcal{R}\subset \rjet X$ a differential relation. 
 $\mathcal{R}$ is said to satisfy the \emph{homotopy principle}\/ 
(or \emph{h-principle}\/ for short) 
if every formal solution to $\mathcal{R}$ is homotopic in $\secc{\mathcal{R}}$ 
to a genuine solution to $\mathcal{R}$. 
 In this paper, we use the notion to show the existence of a genuine solution. 
 In addition to that, 
there are some flavors of the \textit{h}-principle (see~\cite{elmi_book}). 
 In this paper we introduce one of them. 
 The differential relation $\mathcal{R}\subset\ojet{X}$ is said to satisfy 
the \emph{one-parametric \textit{h}-principle}\/ 
if every family $\{f_t\}_{t\in[0,1]}$ of formal solutions to $\mathcal{R}$ 
between genuine solutions $f_0,\ f_1$ 
can be deformed inside $\secc{\mathcal{R}}$ 
to a family $\{\tilde{f}_t\}_{t\in[0,1]}$ of genuine solutions to $\mathcal{R}$ 
keeping both the ends $f_0,\ f_1$. 
 The notion is used to classifications. 

  In general, it is difficult to show 
if a differential relation satisfies the \emph{h}-principles. 
 For each relation, we should apply each suitable method. 
 We introduce one of such methods in the next section. 


\subsubsection{Convex integration method}\label{sec:cvx_itgr}
  In this section, we introduce a method due to Gromov 
to show the \emph{h}-principle 
for certain differential relations, 
which is called the \emph{convex integration} theory. 

  First, we recall ampleness of subsets. 
 Let $P$ be an affine space, and $\Omega\subset P$ a subset. 
 The minimum convex set that include $\Omega$ is called the \emph{convex hull}\/
of $\Omega$. 
 The subset $\Omega\subset P$ is said to be \emph{ample}\/ 
if the convex hull of each path-connected component of $\Omega$ is P. 
 The empty set is also defined to be ample. 

  In order to define the ampleness of differential relations, 
we first introduce the principal directions for fibrations. 
 Let $p\colon X\to V$ be a fibration over an $n$-dimensional manifold $V$ 
with the fiber dimension~$q$. 
 For the $1$-jet $p^1\colon \ojet X\to V$, 
there exists the natural projection $p^1_0\colon\ojet{X}\to X^{(0)}=X$, 
which is an affine bundle. 
 Let 
\begin{equation*}
  E_x:=(p^1_0)^{-1}(x),\quad x\in X
\end{equation*}
denote the fiber of $p^1_0\colon \ojet X\to  X$ over $x\in X$, 
and $\vertt_x\subset T_xX$ the $q$-dimensional tangent space at $x\in X$ 
of the fiber $p^{-1}(p(x))\subset X$ of $p\colon X\to V$ over $p(x)\in V$. 
 Then the fiber $E_x\subset X$ can be identified with 
\begin{equation*}
  \operatorname{Hom}(T_{p(x)}V,\vertt_x)
  \cong\operatorname{Hom}(\mathbb{R}^n,\R^q)
  \cong\{q\times n\ \text{matrices}\} 
\end{equation*}
as follows. 
 As $\ojet{X}$ is the $1$-jet space, a point of $\ojet{X}$ is regarded 
as a non-vertical $n$-dimensional subspace $P_x\subset T_xX$, 
where ``non-vertical'' implies that $P_x$ is transverse 
to $\vertt_x\subset T_xX$. 
 Then, regarding $P_x\subset T_xX\cong T_{p(x)}V\oplus\vertt_x$ as a graph, 
there corresponds a linear mapping. 
 Then, for a fixed point $x\in X$, 
the fiber $E_x\subset\ojet{X}$ is identified with 
$\operatorname{Hom}(T_{p(x)}V, \vertt_x)$. 
 Now we define the principal directions for fibrations. 
 Suppose that a hyperplane $\tau\subset T_{p(x)}V$ 
and a linear mapping $l\colon \tau\to \vertt_x$ are given. 
 For these $\tau$, $l$, 
let $P^l_\tau\subset E_x$ denote the affine subspace of $E_x$ defined as 
\begin{equation*}
  P^l_\tau:=\{L\in\operatorname{Hom}(T_{p(x)}V, \vertt_x)\mid L|_\tau=l\}
  \subset E_x. 
\end{equation*}
 Such affine subspaces of $E_x$ are said to be \emph{principal}. 
 Note that the principal subspaces are parallel $q$-dimensional affine spaces 
if the hyperplane $\tau\subset T_{p(x)}V$ is fixed. 
 The direction of the principal subspaces $P^l_\tau$ determined by $\tau$ 
is called the \emph{principal direction}. 

  Now we define the ampleness of differential relations. 
 Let $p\colon X\to V$ a fibration 
and $\mathcal{R}\subset\ojet{X}$ a differential relation. 
 A differential relation $\mathcal{R}$ is said to be \emph{ample}\/ 
if the intersection of $\mathcal{R}$ with any principal subspaces is ample 
in the affine fiber. 

  Then the key tool in this paper is the following theorem due to Gromov 
(see~\cite{gromov73}, \cite{gromov_pdr}, \cite{elmi_book}, \cite{madachi}, 
\cite{spring}). 

%
%
\begin{thrm}\label{thm:HP_apl_dr}
  Let $p\colon X\to V$ be a fibration and 
$\mathcal{R}\subset\ojet X$ an open differential relation. 
 If $\mathcal{R}\subset \ojet X$ is ample 
then $\mathcal{R}$ satisfies the \textit{h}-principle. 
 In addition, under this condition, $\mathcal{R}$ also satisfies 
the one-parametric \textit{h}-principle. 
\end{thrm} 

%
%
\begin{rmk}
  Gromov showed by this method 
that other flavors of the \textit{h}-principles also holds 
from the ampleness of open differential relations. 
(see~\cite{gromov73}, \cite{elmi_book})
\end{rmk}

  As we see in Section~\ref{sec:diff_rel}, 
singularity theory is a rich source of open differential relations. 
 Let $p\colon X\to V$ a fibration. 
 A singularity $\Sigma\subset\ojet{X}$ is said to be \emph{thin}\/ 
if, at any point $a\in\Sigma$, the intersection $P\cap \Sigma$ 
with any principal subspace $P$ through $a\in\Sigma$ 
is a stratified subset of dimension greater than or equal to $2$ in $P$. 
 When the singularity $\Sigma\subset\ojet{X}$ is thin, 
the open differential relation $\mathcal{R}:=\ojet{X}\setminus\Sigma$ is ample. 
 Then, as a corollary of Theorem~\ref{thm:HP_apl_dr}, we obtain the following. 
%
%
\begin{cor}
  Let $p\colon X\to V$ a fibration. 
 If a singularity $\Sigma\subset\ojet{X}$ is thin, 
then the differential relation $\mathcal{R}:=\ojet{X}\setminus\Sigma$ 
satisfies the \textit{h}-principle. 
 In addition, $\mathcal{R}$ also satisfies 
the one-parametric \textit{h}-principle under the same condition. 
\end{cor} 

  Next, we introduce the homotopy principle for differential sections 
for a linear differential operator $\mathcal{F}$, following~\cite{elmi_book}. 
 Let $p_x\colon X\to V$ and $p_z\colon Z\to V$ be vector bundles 
over a manifold $V$. 
 We first introduce an operator 
$\mathcal{F}\colon \sect{\ojet{X}}\to \sect Z$ 
called a first order linear differential operator as follows. 
 Let $F\colon\ojet{X}\to Z$ be a fiberwise homomorphism 
between vector bundles over $V$. 
 Then there corresponds a mapping $\tilde{F}\colon \sect{\ojet{X}}\to \sect Z$ 
as $\sigma\mapsto F\circ \sigma$ 
for a section $\sigma$ of the jet extension $(p_x)^1\colon\ojet{X}\to V$. 
 By the composition with the differential operator 
$J^1\colon\sect X\to\sect{\ojet{X}}$, 
we have the linear operator 
\begin{equation*}
  \mathcal{F}:=\tilde{F}\circ J^1\colon\sect X\to\sect Z. 
\end{equation*}
 The operators of this type are called 
the \emph{first order linear differential operators}. 
 The bundle homomorphism $F$ 
for a first order linear differential operator $\mathcal{F}$ 
is called the \emph{symbol}\/ of $\mathcal{F}$. 
 Let $\symb\mathcal{F}=F$ denote it. 

  For example, the exterior derivative of differential $p$-forms 
\begin{equation*}
  d\colon\sect{\bigwedge^pT^\ast V}\to \sect{\bigwedge^{p+1}T^\ast V}
\end{equation*}
is a first order linear differential operator. 
 Its symbol $D:=\symb d$ is a fiberwise epimorphism 
$D\colon \ojet{(\bigwedge^pT^\ast V)}\to \bigwedge^{p+1}T^\ast V$. 

  We now introduce the homotopy principle for differential sections. 
 Let $p_x\colon X\to V$ and $p_z\colon Z\to V$ be vector bundles 
over a manifold $V$, 
and $\mathcal{F}\colon \sect X\to \sect Z$ 
a first order linear differential operator. 
 An \emph{$\mathcal{F}$-section}\/ is, by definition, 
a section $s_z\colon V\to Z$ of the bundle $p_z\colon Z\to V$ 
for which there exists a section $s_x\colon V\to X$ of $p_x\colon X\to V$ 
that satisfies $s_z=\mathcal{F}(s_x)$. 
 For example, when $\mathcal{F}=d$ is the exterior derivative 
$d\colon \sect{\bigwedge^{p-1}T^\ast V}\to \sect{\bigwedge^pT^\ast V}$ 
of differential $(p-1)$-forms, 
the $\mathcal{F}$-sections are exact differential $p$-forms. 
 In order to state a property of differential sections, 
we introduce the following notation. 
 For the given subset $S\subset Z$, let $\secc_\mathcal{F}(Z\setminus S)$ denote 
the space of $\mathcal{F}$-sections $s\colon V\to Z$ of $p_z\colon Z\to V$ 
that satisfies $s(V)\subset Z\setminus S$. 
 Then the following theorem concerning the \textit{h}-principle 
for differential sections is introduced in~\cite{elmi_book}. 
%
%
\begin{thrm}\label{thm:HP_D-secs}
  Let $\mathcal{F}\colon \sect X\to \sect Z$ 
be a first order linear differential operator 
for vector bundles $p_x\colon X\to V$ and $p_z\colon Z\to V$ 
over a manifold $V$. 
 Assume that the symbol $F=\symb\mathcal{F}\colon\ojet{X}\to Z$ 
is fiberwise epimorphic. 
 For the given subset $S\subset Z$, set 
\begin{equation*}
  \Sigma:=F^{-1}(S)\subset\ojet{X}.
\end{equation*}
 If the differential relation $\mathcal{R}:=\ojet{X}\setminus\Sigma$ 
satisfies the \textit{h}-principle, 
then the inclusion $\dsec{Z\setminus S}\hookrightarrow\sect{Z\setminus S}$ 
also satisfies the \textit{h}-principle. 
 In other words, any section $s_0\in\sect{Z\setminus S}$ is homotopic 
in $\sect{Z\setminus S}$ to a section $s_1\in\dsec{Z\setminus S}$. 
 In addition, the same claim also holds 
for the one-parametric \textit{h}-principle. 
\end{thrm}

\section{Key theorem and corresponding differential relation}
\label{sec:diff-rel}
  We introduce the key theorem to be proved for the proof of 
Theorems~\ref{thma} and~\ref{thmb} in Subsection~\ref{sec:keythm}. 
  In order to show the theorem by the \emph{h}-principles, 
we define the differential relation for Theorem~\ref{thm:key} 
in Subsection~\ref{sec:def-diffrel}. 
 Then we show the ampleness of the differential relations 
in Subsection~\ref{sec:pf_ethm}. 

\subsection{Key theorem for Theorems~\ref{thma} and~\ref{thmb}}
\label{sec:keythm}
  We introduce the key theorem in this paper. 
 The method to prove the theorem is the \textit{h}-principles. 
 We need to translate the essence of Theorems~\ref{thma} and~\ref{thmb} 
to the philosophy of the \textit{h}-principle. 
 First,  we introduce the formal structure 
for the structures we are dealing with. 
 Then we formulate the theorem to be proved. 

  In terms of the \textit{h}-principle, roughly speaking, 
the formal things are obtained from genuine things 
by forgetting differentiations. 
 The structure dealt in Theorems~\ref{thma} and~\ref{thmb} 
is maximally non-integrable distributions of derived length one and of odd-rank 
(see Section~\ref{sec:distr} for definition). 
 For such structures, the formal structures are introduced as follows. 
 Let $V$ be a manifold of dimension~$n$, 
and $\D$ a distribution of odd rank~$2k+1$. 
 We say $\D$ to be \emph{almost maximally non-integrable}\/ 
if there exists an $(n-2k-1)$-tuple 
$\{\omega_1,\ \omega_2,\dots,\omega_{n-2k-1}\}$ of $2$-forms 
that satisfies that $(\omega_1)^k,\dots,(\omega_{n-2k-1})^k$ are 
linearly independent on $\D$ at each point of $V$. 
 Let 
\begin{equation*}
  \textup{AMNI}(V):=
  \lft\{(\D,\{\omega_1,\dots,\omega_{n-2k-1}\})\;\lft|\; 
  \begin{aligned}
    & \mathcal{D}\ \text{is maximally non-integrable} \\
    & \text{with respect to}\ 2\text{-forms}\ \omega_1,\dots,\omega_{n-2k-1} 
  \end{aligned}
  \rgt.\rgt\}
\end{equation*}
denote the set 
of pairs of an almost maximally non-integrable distribution $\D$
and its defining $2$-forms on $V$. 

  We should remark that if $\D$ is maximally non-integrable 
defined by $(n-2k-1)$-tuple $\{\alpha_1,\dots,\alpha_{n-2k-1}\}$ of $1$-forms, 
then $\{d\alpha_1,\dots,d\alpha_{n-2k-1}\}$ is a tuple of $2$-forms 
that satisfies the condition of the definition above. 
 Then $\D$ is almost maximally non-integrable. 
 If the corank of $\D$ is $n-2k-1=1$, 
an almost maximally non-integrable distribution 
is an almost even-contact structure. 

  By using the notion, almost maximally non-integrable distribution, 
the key theorem is formulated as follows. 

%
%
\begin{thrm}\label{thm:key}
  Let $V$ be a manifold of dimension~$n$, 
and $\D$ a distribution of odd rank~$2k+1$, $k\in\mathbb{N}$, on $V$. 
 Assume that $\D$ is almost maximally non-integrable 
with respect to $(n-2k-1)$-tuple 
$\{\omega_1,\ \omega_2,\dots,\ \omega_{n-2k-1}\}$ of $2$-forms. 
  Then there exists a distribution $\mathcal{E}$ of rank~$2k+1$ on $V$ 
and defining $(n-2k-1)$-tuple $\{\alpha_1,\ \alpha_2,\dots,\ \alpha_{n-2k-1}\}$ 
of $1$-forms for which 
$(\D,\{\omega_i\})$ is homotopic to $(\mathcal{E},\{d\alpha_i\})$ 
in $\textup{AMNI}(V)$. 
 In addition, the one-parametric version also holds. 
 In other words, any path 
$(\mathcal{D}_t,\{\omega^t_i\})\in\textup{AMNI}(V)$, $t\in[0,1]$, of 
almost maximally non-integrable distributions of derived length one 
between two genuine maximally non-integrable distributions 
$\mathcal{D}_0,\ \mathcal{D}_1$ of derived length one 
can be deformed in $\textup{AMNI}(V)$ 
to a path $(\tilde{\mathcal{D}}_t,\{d\alpha^t_i\})$ of genuine 
maximally non-integrable distributions of derived length one 
keeping both $\mathcal{D}_0=\tilde{\mathcal{D}}_0$ 
and $\mathcal{D}_1=\tilde{\mathcal{D}}_1$. 
\end{thrm} 
\noindent
 This theorem is proved in the next section 
by using the \textit{h}-principles introduced in Section~\ref{sec:h-prin}. 
 Theorem~\ref{thma} and Theorem~\ref{thmb} follows this theorem 
(see Section~\ref{sec:proof}). 

\subsection{Differential relations for the problems}\label{sec:def-diffrel}
  In order to show Theorem~\ref{thm:key}, 
we apply Gromov's convex integration method, Theorem~\ref{thm:HP_apl_dr}, 
and the \textit{h}-principle for differential sections, 
Theorem~\ref{thm:HP_D-secs}. 
  Then, in this subsection, 
we define the first order linear differential operator 
and the differential relations 
concerning the problems dealt in Theorem~\ref{thm:key} 
in order to apply the theorems above. 

  First, we define a first order linear differential operator 
so that the target is concerning Theorem~\ref{thm:key}, 
in order to apply Theorem~\ref{thm:HP_D-secs}. 
 Let $V$ be a manifold of dimension~$n$. 
 Let $X\to V$ and $Z\to V$ be vector bundles over $V$ defined as 
\begin{equation}\label{eq:XZbdls}
  \begin{aligned}
    X&:=\bigoplus^{n-2k-1} T^\ast V\to V, \\
    Z&:=\lft(\bigoplus^{n-2k-1} T^\ast V\rgt)\oplus\lft(\bigwedge^2 T^\ast V\rgt)
    \to V, 
  \end{aligned}
\end{equation}
where $k\in\mathbb{N}$ is the same integer as in the statement 
of Theorem~\ref{thm:key}. 
 Then we define the first order linear differential operator 
\begin{equation*}
  \mathcal F\colon\sect X\to\sect{\bigoplus^{n-2k-1} Z}
\end{equation*}
as follows. 
 Let $\mathcal{F}_i\colon\sect X\to\sect Z$, $i=1,2,\dots,n-2k-1$, 
be a first order linear differential operator defined as 
$\mathcal{F}_i\colon (\alpha_1,\alpha_2,\dots,\alpha_{n-2k-1})\mapsto 
(\alpha_1,\alpha_2,\dots,\alpha_{n-2k-1}, d\alpha_i)$. 
 Then we define the operator $\mathcal{F}$ as
\begin{align}\label{eq:def_f}
  \mathcal{F}:=(\mathcal{F}_1,\dots,\mathcal{F}_{n-2k-1})\colon
  &(\alpha_1,\dots,\alpha_{n-2k-1})\mapsto \notag \\
  &\bigl((\alpha_1,\dots,\alpha_{n-k-1},d\alpha_1),\dots,
  (\alpha_1,\dots,\alpha_{n-k-1},d\alpha_{n-k-1})\bigr). 
\end{align}
 We should remark that the symbol $F\colon X^{(1)}\to\bigoplus^{n-2k-1}Z$ 
of $\mathcal{F}$ is fiberwise epimorphic. 

  Next, we define a differential relation in the source side of $\mathcal{F}$ 
that corresponds to a critical condition for Theorem~\ref{thm:key} 
in the target side. 
 Let $S\subset\bigoplus^{n-2k-1} Z$ be the subset defined as
 \begin{align*}
   S:=&\lft\{\lft((\alpha^1_1,\dots,\alpha^1_{n-2k-2},\omega_1)_v,\dots,
        (\alpha^{n-2k-1}_1,\dots,\alpha^{n-2k-1}_{n-2k-2},\omega_{n-2k-1})_v\rgt)
        \rgt. \\
      &\hspace{5cm}\in\bigoplus^{n-2k-1}\lft(\bigoplus^{n-2k-1}T^\ast V
        \oplus\bigwedge^2T^\ast V\rgt)=\bigoplus^{n-2k-1}Z \\
      &\quad\big|\; 
     (\alpha^1_1\wedge\dots\wedge\alpha^1_{n-2k-1}\wedge(\omega_1)^k)_v,\dots, 
        (\alpha^{n-2k-1}_1\wedge\dots\wedge\alpha^{n-2k-1}_{n-2k-1}
        \wedge(\omega_1)^k)_v \\
      &\hspace{10cm} \text{are linearly dependent}\bigr\}. 
 \end{align*}
 In other words, in the fiber $\lft(\bigoplus^{n-2k-1} Z\rgt)_v$, 
$S$ is defined by the condition that 
$(\alpha^1_1\wedge\dots\wedge\alpha^1_{n-2k-1}\wedge(\omega_1)^k)_v,\dots, 
        (\alpha^1_1\wedge\dots\wedge\alpha^1_{n-2k-1}\wedge(\omega_1)^k)_v$ 
are linearly dependent. 
 Then we take the inverse image of $S\subset\bigoplus^{n-2k-1}Z$ 
by the symbol $F\colon X^{(1)}\to\bigoplus^{n-2k-1}Z$ 
of $\mathcal{F}\colon \sect X\to \sect{\bigoplus^{n-2k-1}Z}$ 
(see Section~\ref{sec:h-prin} for definition). 
 Set $\Sigma:=F^{-1}(S)\subset X^{(1)}$. 
 We remark that, from the construction of $\mathcal{F}$, 
the subset $\Sigma\subset X^{(1)}$ is the same 
as the inverse image $F^{-1}(\tilde{S})\subset X^{(1)}$ of the set 
\begin{align}\label{eq:def_s}
     \tilde{S}:=&\lft\{\lft((\alpha_1,\dots,\alpha_{n-2k-2},\omega_1)_v,\dots,
        (\alpha_1,\dots,\alpha_{n-2k-2},\omega_{n-2k-1})_v\rgt)
        \rgt. \notag \\
      &\hspace{5cm}\in\bigoplus^{n-2k-1}\lft(\bigoplus^{n-2k-1}T^\ast V
        \oplus\bigwedge^2T^\ast V\rgt)=\bigoplus^{n-2k-1}Z \notag \\
      &\quad\big|\; 
     (\alpha_1\wedge\dots\wedge\alpha_{n-2k-1}\wedge(\omega_1)^k)_v,\dots, 
        (\alpha_1\wedge\dots\wedge\alpha_{n-2k-1}\wedge(\omega_{n-2k-1})^k)_v 
        \notag \\
      &\hspace{8cm} \text{are linearly dependent}\bigr\}. 
\end{align}
 Then as the subset $\tilde{S}\subset\bigoplus^{n-2k-1}$ corresponds to where 
the maximally non-integrability collapses, 
the subset $\Sigma=F^{-1}(\tilde{S})\subset X^{(1)}$ is regarded 
as the singularity to be considered. 
 Set 
\begin{equation}\label{eq:def_dr}
\mathcal{R}:=X^{(1)}\setminus\Sigma\subset X^{(1)}. 
\end{equation}
 It is the open differential relation to be considered in the next section 
to show Theorem~\ref{thm:key}.

\section{Proof of Theorem~\ref{thm:key}}\label{sec:pf_ethm}
  In this section, we show Theorem~\ref{thm:key}. 
 As mentioned above, we apply Gromov's convex integration method, 
and the \textit{h}-principle for differential sections. 
 Then the key of the proof is the following proposition. 
 The notations introduced in the previous section are used. 
 Let $V$ be an $n$-dimensional manifold, 
$X=\bigoplus^{n-2k-1}T^\ast V$ 
and $Z=\lft(\bigoplus^{n-2k-1}T^\ast V\rgt)\oplus
\lft(\bigoplus^k\lft(\bigwedge^2 T^\ast V\rgt)\rgt)$ 
the vector bundles over $V$,
$\mathcal{F}\colon\sect X\to\sect{\bigoplus^{n-2k-1}Z}$ 
the linear differential operator, 
and $\mathcal{R}=X^{(1)}\setminus\Sigma\subset X^{(1)}$ 
the differential relation. 
%
%
\begin{prop}\label{prop:core}
  The open differential relation $\mathcal{R}\subset X^{(1)}$ is ample. 
\end{prop} 

\begin{proof}
  In order to show that the differential relation 
$\mathcal{R}=\ojet{X}\setminus\Sigma\subset\ojet{X}$ is ample, 
we will show that the singularity $\Sigma\subset\ojet{X}$ is thin. 

  It is enough that we discuss by using local coordinates 
of the base manifold $V$ of the bundles. 
 Let $(x_1,x_2,\dots,x_n)$ be local coordinates of $V$. 
 Then $\{(dx_1)_v,\dots,(dx_n)_v\}$ is a basis of the fiber $T^\ast_vV$. 
 Let $(a_1,\dots,a_n)$ be coordinates on the fiber of $T^\ast V$. 
 Similarly, $\{(dx_i\wedge dx_j)_v\}_{1<i<j<n}$ is a basis of the fiber 
$\lft(\bigwedge^2 T^\ast V\rgt)_v$. 
 Let $(z_{12},\dots,z_{(n-1)n})$ be coordinates on the fiber 
of $\bigwedge^2 T^\ast V$. 
 Further, let $(a_1,\dots,a_n,y_{11},y_{12},\dots,y_{nn})$ be coordinates of 
the fiber $(T^\ast V)^{(1)}_v$, 
where $y_{ij}$ corresponds to $\rd a_i/\rd x_j$. 

  By using such local coordinates, we write down the singularity 
$\Sigma\subset \ojet{\lft(\bigoplus^{n-2k-1}T^\ast V\rgt)}=\ojet{X}$ as follows. 
 Recall that $\Sigma$ is defined as the inverse image 
$\Sigma=F^{-1}(\tilde{S})$ 
by the symbol $F\colon\ojet{X}\to\sect{\bigoplus^{n-2k-1}Z}$ 
of the linear differential operator 
$\mathcal{F}\colon\sect X\to\sect{\bigoplus^{n-2k-1}Z}$ 
defined by Equation~\eqref{eq:def_f} concerning exterior derivative 
of differential forms. 
 Note that, by the exterior derivative, 
the coordinates $y_{ij}-y_{ji}$ correspond to $z_{ij}$. 
 Recall that the set $\tilde{S}\subset\sect{\bigoplus^{n-2k-1}Z}$ 
is defined by the following condition (see Equation~\eqref{eq:def_s}): 
$(n-2k-1)$ pieces of $(n-1)$-forms 
\begin{equation}\label{eq:dep_forms}
  \alpha_1\wedge\dots\wedge\alpha_{n-2k-1}\wedge(\omega_1)^k,\dots, 
  \alpha_1\wedge\dots\wedge\alpha_{n-2k-1}\wedge(\omega_{n-2k-1})^k
\end{equation}
are linearly independent on each fiber, 
where $\alpha_i\in\sect{T^\ast V}$ are $1$-forms 
and $\omega_i\in\sect{\bigwedge^2 T^\ast V}$ are $2$-forms on $V$. 
 Then, in order to wright down the singularity 
$\Sigma=F^{-1}(\tilde{S})\subset \ojet{X}$ by using local coordinates, 
we represent $\alpha_i$ and $\omega_i$ on a fiber over $v\in V$ 
by such coordinates as follows: 
\begin{align*}
  \alpha_i&=\sum_{j=1}^n a^i_jdx_j 
            =a^i_1dx_1+a^i_2dx_2+\dots+a^i_ndx_n, \quad (i=1,2,\dots,n-2k-1), \\
  \omega_i&=\sum_{1\le j\le k\le n}z^i_{jk}dx_j\wedge dx_k \\
          & =z^i_{12}dx_1\wedge dx_2+\dots+z^i_{(n-1)n}dx_{n-1}\wedge dx_n, 
            \qquad (i=1,2,\dots,n-2k-1). 
\end{align*}
 Following this representation, the $(n-1)$-forms 
in Equation~\eqref{eq:dep_forms} are written down as follows. 
 First, we have 
\begin{equation*}
  (\omega_i)^k=\sum_{1\le j_1<\dots<j_{2k}\le n}(A^i_{j_1\dots j_{2k}})
  dx_{j_1}\wedge\dots\wedge dx_{j_{2k}}, \quad i=1,2,\dots,n-2k-1, 
\end{equation*}
where the coefficients are 
\begin{align}\label{eq:def_Aij}
  A^i_{j_1\dots j_{2k}}=&\sum_{\{l_1,\dots,l_{2k}\}=\{j_1,\dots,j_{2k}\}}
  \sigma(l_1,\dots,l_{2k})\cdot z^i_{l_1l_2}\cdots z^i_{l_{2k-1}l_{2k}}, \\
  &\qquad i=1,2,\dots,n-2k-1,\quad 1\le j_1<j_2<\dots<j_{2k}\le n, \notag
\end{align}
for the sign 
$\sigma(l_1,\dots,l_{2k})=\operatorname{sgn}
\begin{pmatrix}
  j_1&\cdots&j_{2k}\\l_1&\cdots&l_{2k}
\end{pmatrix}$ 
of permutations. 
 Note that $z^i_{jk}$ is valid for $j<k$. 
 Then the $(n-1)$-forms in Equation~\eqref{eq:dep_forms} are 
\begin{align}\label{eq:loc_dep_forms}
  \alpha_1\wedge\alpha_2\wedge\dots\wedge\alpha_{n-2k-1}\wedge(\omega_i)^k
  =&\sum_{r=1}^n(B^i_r) dx_1\wedge\dots\wedge\widehat{dx_r}\wedge
     \dots\wedge dx_n, \\
   &\qquad i=1,2,\dots,n-2k-1, \notag
\end{align}
where ``$\widehat{dx_r}$'' implies ``without $dx_r$.'' 
 The coefficients are 
\begin{align}\label{eq:def_Bij}
  B^i_r=&\sum_{\{p_1,\dots,p_{n-1}\}=\{1,\dots,\widehat{r},\dots,n\}}
          \sigma(p_1,\dots,p_{n-1})\cdot a^1_{p_1}\cdots a^{n-2k-1}_{p_{n-2k-1}}
          \cdot\lft(A^i_{p_{n-2k}\dots p_{n-1}}\rgt), \\
        &\qquad i=1,2,\dots,n-2k-1,\quad r=1,2,\dots,n. \notag
\end{align}
 Note that $A^i_{p_{n-2k}\dots p_{n-1}}$ is valid 
for $p_{n-2k}<p_{n-2k+1}<\dots<p_{n-1}$. 
 At last, we obtain a representation of the singularity $\Sigma\subset\ojet{X}$ 
by the local coordinates. 
 From the expression in Equation~\eqref{eq:loc_dep_forms}, 
the defining condition of $\Sigma\subset\ojet{X}$ is 
that the following $n-2k-1$ vectors 
\begin{equation*}
  \beta_i:=(B^i_1,B^i_2,\dots,B^i_n),\qquad i=1,2,\dots,n-2k-1, 
\end{equation*}
are linearly independent. 
 In other words, there exists $(c_1,\dots,c_{n-2k-1})$ vanishing nowhere 
that satisfies
\begin{equation}\label{eq:l-indep_condi}
  c_1\beta_1+c_2\beta_2+\dots+c_{n-2k-1}\beta_{n-2k-1}=0. 
\end{equation}

  In order to show that the singularity $\Sigma\subset\ojet{X}$ is thin, 
we observe the intersections of $\Sigma$ with principal subspaces, 
using local coordinates. 
 To make the discussion simple, we take a principal direction of $P$ 
that concerns $x_1$ in the local coordinates of the base manifold $V$. 
 In other words, in the defining condition of $\Sigma\subset\ojet{X}$, 
only $z^i_{1l}$ are variables. 
 Other $z^i_{ml}$ and $a^i_j$ are constant on the intersection 
in a fiber over $v\in V$. 
 This implies, from Equations~\eqref{eq:def_Aij} and~\eqref{eq:def_Bij}, 
that $B^i_1$, $i=1,2,\dots,n-2k-1$, are constant, and 
\begin{align}\label{eq:loc_coeff}
  B^i_r=&\overline{C^i_r}+\sum_{\{p_1,\dots,p_{n-2k}=1,\dots,p_{n-1}\}=\{1,\dots,\widehat{r},\dots,n\}}
  \sigma(p_1,\dots,p_{n-1})
  \cdot\lft(A^i_{1(p_{n-2k+1})\dots (p_{n-1})}\rgt), \notag \\
  =&\overline{C^i_r}+\sum_{m=2}^n (C^i_r(m))\cdot z^i_{1m}, \notag \\
  =&\overline{C^i_r}+(C^i_r(2))\cdot z^i_{12}+(C^i_r(3))\cdot z^i_{13}+\dots+
     (C^i_r(r))\cdot z^i_{1r}+\dots+(C^i_r(n))\cdot z^i_{1n}, \\
  &\qquad i=1,2,\dots,n-2k-1,\quad r=2,3,\dots,n, \notag
\end{align}
where $\overline{C^i_r}$ and 
\begin{align*}
  C^i_r(r)&=0,  \\
  C^i_r(m)&=k\sum_{\{p_1,\dots,p_{n-3}\}=\{1,\dots,n\}\setminus\{1,r,m\}}
            a^i_{p_1}\cdots a^i_{p_{n-2k-1}}A^i_{p_{n-2k}\cdots p_{n-3}}
\end{align*}
are constant. 
 We should remark here that 
\begin{equation}\label{eq:pseudosymm}
  C^i_r(m)=\pm C^i_m(r). 
\end{equation}

  Now, we observe the intersection of the singularity $\Sigma\subset\ojet{X}$ 
with the principal subspace $P$ with respect to $x_1$ direction, 
by using local coordinates. 
 First, we confirm the system of linear equations 
that represents $\Sigma\cap P$. 
 Recall that $\Sigma\subset\ojet{X}$ is defined 
by the condition~\eqref{eq:l-indep_condi}. 
 In other words, 
\begin{equation*}
  \lft\{
  \begin{matrix}
    c_1B^1_1+c_2B^2_1+\dots+c_{n-2k-1}B^{n-2k-1}_1=0, \\
    c_1B^1_2+c_2B^2_2+\dots+c_{n-2k-1}B^{n-2k-1}_2=0, \\
    \vdots \\
    c_1B^1_n+c_2B^2_n+\dots+c_{n-2k-1}B^{n-2k-1}_n=0. \\
  \end{matrix}
  \rgt.
\end{equation*} 
 Note that, since $B^i_1$, $i=1,2,\dots,n-2k-1$, are constant, 
$c_1,c_2,\dots,c_{n-2k-1}$ should satisfy the first equation. 
 If there are no such $c_1,c_2,\dots,c_{n-2k-1}$, 
there is no intersection between $\Sigma$ and $P$ on the fiber $\ojet{X}_v$ 
over $v\in V$. 
 Then, from the second to $n$-th equations, 
and by the Equation~\eqref{eq:loc_coeff}, 
the intersection of $\Sigma$ with $P$ is defined 
by the following system of linear equations 
with respect to $(n-1)\times (n-2k-1)$ variables 
$z^1_{12},\dots,z^1_{1n},z^2_{12},\dots,z^{n-2k-1}_{1n}$:
\begin{equation*}
  \lft\{
  \begin{aligned}
    c_1C^1_2(3)z^1_{13}+c_1C^1_2(4)z^1_{14}+\dots+c_1C^1_2(n)z^1_{1n}& \\
    +c_2C^2_2(3)z^2_{13}+\dots+c_{n-2k-1}C^{n-2k-1}_2(n)z^{n-2k-1}_{1n}&
    =-c_1\overline{C^1_2}-c_2\overline{C^2_2}-\dots-c_n\overline{C^n_2}\\
    c_1C^1_3(2)z^1_{12}+c_1C^1_3(4)z^1_{14}+\dots+c_1C^1_3(n)z^1_{1n}& \\
    +c_2C^2_2(2)z^2_{12}+\dots+c_{n-2k-1}C^{n-2k-1}_2(n)z^{n-2k-1}_{1n}&
    =-c_1\overline{C^1_2}-c_2\overline{C^2_2}-\dots-c_n\overline{C^n_2}\\
    &\vdots \\
    c_1C^1_n(2)z^1_{12}+c_1C^1_n(3)z^1_{13}+\dots+c_1C^1_n(n-1)z^1_{1(n-1)}& \\
    +c_2C^2_n(2)z^2_{12}+\dots+c_{n-2k-1}C^{n-2k-1}_n(n-1)z^{n-2k-1}_{1(n-1)}&
    =-c_1\overline{C^1_n}-c_2\overline{C^2_n}-\dots-c_n\overline{C^n_n}. 
  \end{aligned}
  \rgt.
\end{equation*}

  Then we observe the codimension of $\Sigma\cap P\subset P$ 
to show $\Sigma\subset\ojet{X}$ is thin. 
 The rank of the coefficient matrix of the system of equations above 
is the codimension of $\Sigma\cap P\subset P$. 
 The coefficient matrix is the $(n-1)\times(n-1)(n-2k-1)$-matrix written down as
\begin{equation*}
  \begin{pmatrix}
    0&\tilde{C}^1_2(3)&\tilde{C}^1_2(4)&\dots&\tilde{C}^1_2(n)&0&\tilde{C}^2_2(3)&\dots&\tilde{C}^{n-2k-1}_2(n)\\
    \tilde{C}^1_3(2)&0&\tilde{C}^1_3(4)&\dots&\tilde{C}^1_3(n)&\tilde{C}^2_3(2)&0&\dots&\tilde{C}^{n-2k-1}_3(n)\\
    \vdots &\vdots&\vdots&&\vdots&\vdots&\vdots&&\vdots \\
    \tilde{C}^1_n(2)&\tilde{C}^1_n(3)&\tilde{C}^1_3(4)&\dots&0&\tilde{C}^2_n(2)&\tilde{C}^2_n(3)&\dots&0\\
  \end{pmatrix}, 
\end{equation*}
where $\tilde{C}^i_r(m)=c_iC^i_r(m)$. 
 When the rank of this matrix is $0$, the differential relation 
$\mathcal{R}=\ojet{X}\setminus\Sigma$ is ample by definition. 
 We claim that the rank of this matrix never be $1$. 
 Suppose the $i$-th column $\boldsymbol{a}_i\in\mathbb{R}^{n-1}$ is non-zero. 
Assume $i\equiv\bar{i} \pmod{n-1}$ for $1\le\bar{i}\le n-1$, 
and $i=\bar{i}+j(n-1)$. 
 Then the $\bar{i}$-th element of $\boldsymbol{a}_i$ is 
$\tilde{C}^{j+1}_{\bar{i}+1}(\bar{i}+1)=c_{j+1}C^{j+1}_{\bar{i}+1}(\bar{i}+1)=0$. 
 As $\boldsymbol{a}_i\in\mathbb{R}^{n-1}$ is not zero, 
there should be a non-zero element, 
say $\tilde{C}^{j+1}_{r}(\bar{i}+1)\ne 0$ for $r\ne \bar{i}+1$. 
 From Equation~\eqref{eq:pseudosymm}, 
we have $\tilde{C}^{j+1}_{\bar{i}+1}(r)=\tilde{C}^{j+1}_{r}(\bar{i}+1)\ne 0$. 
 This implies that the $(r-1+j(n-1))$-st column 
$\boldsymbol{a}_{r-1+j(n-1)}\in\mathbb{R}^{n-1}$ of the coefficient matrix 
is not zero, 
and the $\bar{i}$-th element of $\boldsymbol{a}_{r-1+j(n-1)}$ 
is $\tilde{C}^{j+1}_{\bar{i}+1}(r)\ne 0$. 
 Therefore, $\boldsymbol{a}_i, \boldsymbol{a}_{r-1+j(n-1)}\in\mathbb{R}^{n-1}$ 
are linearly independent. 
 Thus we conclude that the rank of the coefficient matrix is greater than $1$. 

  Thus, we have proved that the singularity $\Sigma$ is thin 
and the differential relation $\mathcal{R}=\ojet{X}\setminus\Sigma$ is ample. 
\end{proof}

  In the rest of this section, we prove Theorem~\ref{thm:key}. 

  \begin{proof}[Proof of Theorem~\textup{\ref{thm:key}}]
 Let $V$ be a manifold of dimension~$n$, 
and $\mathcal{D}$ a distribution of rank~$2k+1$ on $V$. 
 Suppose that the $1$-forms $\alpha_i$ and $2$-forms $\omega_i$ satisfy 
the assumptions in Theorem~\ref{thm:key} for $\mathcal{D}$. 
 We prove in the following two steps. 

  First, we apply Proposition~\ref{prop:core} and 
Gromov's \textit{h}-principle for ample differential relations 
(Theorem~\ref{thm:HP_apl_dr}). 
 Let the vector bundle $X:=\bigoplus^{n-2k-1}T^\ast V$ over $V$ 
be as in Equations~\eqref{eq:XZbdls}. 
 Now, we take the differential relation 
$\mathcal{R}\subset\ojet{X}$ to be considered as in Equation~\eqref{eq:def_dr}. 
 Then, from Proposition~\ref{prop:core}, $\mathcal{R}\subset\ojet{X}$ is ample. 
 This implies, according to Theorem~\ref{thm:HP_apl_dr}, 
that the differential relation $\mathcal{R}\subset\ojet{X}$ satisfies 
the \textit{h}-principle and the one-parametric \textit{h}-principle. 

  Next, we discuss the \textit{h}-principles for differential operators. 
 For the vector bundle $X$ above, let the vector bundle 
$Z:=\lft(\bigoplus^{n-2k-1}T^\ast V\rgt)\oplus\lft(\bigwedge^2T^\ast V\rgt)$ 
over $V$ be as in Equations~\eqref{eq:XZbdls}, 
and the linear differential operator 
$\mathcal{F}\colon\sect{X}\to\sect{\bigoplus^{n-2k-1}Z}$ 
be as in Equation~\eqref{eq:def_f}. 
 Note that the symbol $F\colon X^{(1)}\to\bigoplus^{n-2k-1}Z$ of $\mathcal{F}$ 
is fiberwise epimorphic. 
 Recall that the differential relation $\mathcal{R}\subset\ojet{X}$ is defined 
as $\mathcal{R}=\ojet{X}\setminus\Sigma=\ojet{X}\setminus F^{-1}(\tilde{S})$ 
for the set $\tilde{S}\subset\bigoplus^{n-2k-1}Z$ 
defined in Equation~\eqref{eq:def_s}. 
 Now, we have proved that the differential relation $\mathcal{R}\subset\ojet{X}$
satisfies the \textit{h}-principle and the one-parametric \textit{h}-principle. 
 Then, from Theorem~\ref{thm:HP_D-secs}, 
we obtain that the \textit{h}-principle 
and the one-parametric \textit{h}-principle hold for the inclusion 
\begin{equation*}
  \secc_{\mathcal{F}}\lft(\lft(\bigoplus^{n-2k-1}Z\rgt)\setminus\tilde{S}\rgt)
  \hookrightarrow 
  \sect{\lft(\bigoplus^{n-2k-1}Z\rgt)\setminus\tilde{S}}, 
\end{equation*}
where
\begin{align*}
  &\secc_{\mathcal{F}}\lft(\lft(\bigoplus^{n-2k-1}Z\rgt)\setminus\tilde{S}\rgt)\\
  :=&
  \lft\{\lft.s\colon V\to\lft(\bigoplus^{n-2k-1}Z\rgt)\setminus\tilde{S},\ 
  \begin{aligned}
    &\text{section of} \\
    &\lft(\bigoplus^{n-2k-1}Z\rgt)\setminus\tilde{S}
  \end{aligned}
  \;\rgt|\; 
  \begin{aligned}
    &s=\mathcal{F}(t)\\ 
    &\text{for some section}\ t\ \text{of}\ X
  \end{aligned}
  \rgt\}
\end{align*}
is the space of $\mathcal{F}$-sections. 
 Since 
$\secc_{\mathcal{F}}\lft(\lft(\bigoplus^{n-2k-1}Z\rgt)\setminus\tilde{S}\rgt)$
corresponds to the set of genuine maximally non-integrable distributions 
and $\sect{\lft(\bigoplus^{n-2k-1}Z\rgt)\setminus\tilde{S}}$ 
corresponds to the set of almost maximally non-integrable distributions, 
this implies Theorem~\ref{thm:key}. 
\end{proof}

\section{Proof of Theorems~\ref{thma} and~\ref{thmb}  and further observations}
\label{sec:proof}
  Following Theorem~\ref{thm:key}, we prove Theorems~\ref{thma} and~\ref{thmb}
in Subsection~\ref{sec:proofs}. 
 Further, we discuss some miscellaneous things 
on the case when the tangent distributions are orientable. 
 Corollary~\ref{corA} is proved in Subsection~\ref{sec:ori}. 

\subsection{Proof of theorems}\label{sec:proofs}
  First, we show Theorem~\ref{thma}. 
 Since a genuine maximally non-integrable distribution of derived length one 
of odd rank 
is an almost maximally non-integrable distribution of derived length one, 
the necessity of Theorem~\ref{thma} is clear. 
 Then we show the sufficiency of the theorem. 
 Let $V$ be a manifold of dimension~$n$, 
and $\mathcal{D}\subset TV$ a distribution of rank~$r=2k+1$. 
 Suppose that it is almost maximally non-integrable of derived length one. 
 In other words, there is an $(n-2k-1)$-tuple of $2$-forms 
$\{\omega_1,\omega_2,\dots,\omega_{n-2k-1}\}$ for which 
$(\omega_1)^k,\dots,(\omega_{n-2k-1})^k$ are pointwise linearly independent 
on $\mathcal{D}$. 
 Then, from Theorem~\ref{thm:key}, 
on the manifold $V$ there exists an almost maximally non-integrable distribution
$\mathcal{E}=\{\alpha_1=0,\alpha_2=0,\dots,\alpha_{n-2k-1}=0\}\subset TV$ 
of rank~$2k+1$ of derived length one 
with respect to $(n-2k-1)$-tuple 
$\{d\alpha_1,\ d\alpha_2,\dots,d\alpha_{n-2k-1}\}$ of $2$-forms. 
 This implies that $\mathcal{E}\subset TV$ 
is a genuine maximally non-integrable distribution of derived length one on $V$.
 This conclude the proof of Theorem~\ref{thma}. 

  Then, we show Theorem~\ref{thmb}. 
 We apply the one-parametric \textit{h}-principle in Theorem~\ref{thm:key}. 
 Let $\mathcal{D}_0,\ \mathcal{D}_1\subset TV$ be maximally non-integrable 
distributions of derived length one of odd rank on a manifold $V$. 
 Suppose that they are homotopic 
through almost maximally non-integrable distributions of derived length one. 
 In other words, there is a path $(\mathcal{D}_t,\{\omega_i^t\})$ 
of almost maximally non-integrable distributions of derived length one 
between them. 
 Then, from Theorem~\ref{thm:key}, the path $(\mathcal{D}_t,\{\omega_i^t\})$ 
can be deformed to a path $(\tilde{\mathcal{D}}_t,\{d\alpha_i^t\})$ 
of genuine maximally non-integrable distributions of derived length one 
keeping both $\mathcal{D}_0=\tilde{\D}_0$, $\mathcal{D}_1=\tilde{\D}_1$. 
 This implies that $\mathcal{D}_0$ and $\mathcal{D}_1$ are isotopic. 
 This conclude the proof of Theorem~\ref{thmb}. \qed

\subsection{Orientable cases}\label{sec:ori}
 In this subsection, we only deal with orientable tangent distributions. 
 In some cases, the topological condition on manifolds 
in the existence theorem, Theorem~\ref{thma}, can be described simpler. 
 First, we introduce more general description. 
 Then we obtain Corollary~\ref{corA} as a special case. 

  For orientable distributions, 
fixing dimension of manifolds and rank of tangent distributions, 
we discuss the existence of maximally non-integrable distributions 
of derived length one. 
 Recall that, for such a distribution of rank~$2k+1$, 
the dimension~$n$ of manifolds should be $2k+2\le n\le 2(2k+1)=4k+2$ 
(see Section~\ref{sec:distr}). 
 When corank of distribution is large, 
we obtain the following proposition as a corollary of Theorem~\ref{thma}. 
 In other words, we obtain the condition on manifolds 
to have a maximally non-integrable orientable distribution 
of type~$(2k+1,4k+1)$ or~$(2k+1,4k+2)$. 
%
%
\begin{prop}\label{prop:ori}
  Let $V$ be a manifold possibly closed. 
 $V$ admits a maximally non-integrable orientable distribution 
of type~$(2k+1,4k+1)$ or~$(2k+1,4k+2)$ 
if and only if 
there exists a trivial subbundle $\mathcal{E}\subset TV$ of rank~$2k+1$. 
\end{prop} 

\begin{proof}
  First, we discuss the sufficiency. 
 Suppose that there exists a trivial subbundle $\mathcal{E}\subset TV$ of 
rank~$2k+1$ 
on a manifold $V$ of dimension~$n=4k+1$ or~$4k+2$. 
 Then there exists a coframe $X^\ast_1,\dots,X^\ast_{2k+1}$ 
of $\mathcal{E}^\ast\subset T^\ast V$. 
 Setting for each $i=1,2,\dots,2k+1$, 
\begin{equation*}
  \omega_i:=
  X^\ast_{p_1}\wedge X^\ast_{p_2}+\dots+X^\ast_{p_{2k-1}}\wedge X^\ast_{p_{2k}}, 
  \qquad \{p_1,\dots,p_{2k}\}=\{1,\dots,\hat{i},\dots,2k+1\}, 
\end{equation*}
we have $(2k+1)$-tuple $\{\omega_1,\dots,\omega_{2k+1}\}$ of $2$-forms. 
 This tuple of $2$-forms satisfies 
that $(\omega_1)^k,\dots,(\omega_{2k+1})^k$ are linearly independent $2k$-forms 
on the distribution $\mathcal{E}$ of rank~$2k+1$ since
\begin{equation*}
  (\omega_i)^k=\pm(k!) X^\ast_1\wedge\dots\wedge\widehat{X^\ast_i}\wedge\dots
  \wedge X^\ast_{2k+1},\quad i=1,2,\dots,2k+1. 
\end{equation*}
 When the dimension of $V$ is $n=4k+1$ (resp.\ $n=4k+2$), 
the corank of $\mathcal{E}\subset TV$ is $2k$ (resp.\ $2k+1$). 
 Then $\mathcal{E}$ with $\{\omega_1,\dots,\omega_{2k}\}$ 
(resp.\ $\{\omega_1,\dots,\omega_{2k+1}\}$) 
is almost maximally non-integrable of derived length one. 
 Then, from Theorem~\ref{thma}, the manifold $V$ admits 
a maximally non-integrable distribution $\mathcal{D}\subset TV$ 
of derived length one of rank~$2k+1$. 
 In other words, $\mathcal{D}\subset TV$ 
is a maximally non-integrable distribution 
of type~$(2k+1,4k+1)$ or~$(2k+1,4k+2)$. 

  Next, we discuss the necessity. 
 Suppose that a manifold $V$ admits a maximally non-integrable 
orientable distribution of type~$(2k+1,4k+1)$ or $(2k+1,4k+2)$. 
 From Theorem~\ref{thma}, we may suppose 
that there exists an almost maximally non-integrable orientable distribution 
$\mathcal{D}\subset TV$ of derived length one of rank~$2k+1$. 
 When the dimension of the manifold $V$ is $n=4k+1$ (resp.\ $n=4k+2$), 
the codimension of $\mathcal{D}\subset TV$ is $2k$ (resp.\ $2k+1$). 
 Then there exists a $(2k)$-tuple of $2$-forms $\{\omega_1,\dots,\omega_{2k}\}$ 
(resp.\ $(2k+1)$-tuple of $2$-forms $\{\omega_1,\dots,\omega_{2k+1}\}$) 
for the almost maximal non-integrability (see Section~\ref{sec:distr}). 
 This implies that $\{(\omega_1)^k,\dots,(\omega_{2k})^k\}$ is a $2k$-tuple 
(resp.\ $\{(\omega_1)^k,\dots,(\omega_{2k+1})^k\}$ is a $(2k+1)$-tuple) 
of linearly independent $(2k)$-forms 
on the distribution $\mathcal{D}$ of rank~$2k+1$. 

  Now we observe the tangent distribution $\mathcal{D}\subset TV$. 
 Since $\mathcal{D}$ is orientable, $\bigwedge^{2k+1}\mathcal{D}$ is trivial. 
 Then $\bigwedge^{2k+1}\mathcal{D}^\ast$ is also trivial. 
 Taking a non-zero section, or $(2k+1)$-form, 
$\gamma$ of $\bigwedge^{2k+1}\mathcal{D}^\ast$, 
we have an equivalence $\mathcal{D}\cong\bigwedge^{2k}\mathcal{D}^\ast$ 
by the contraction $X\mapsto \iota_X \gamma$. 
 Then $\bigwedge^{2k}\mathcal{D}^\ast$ is also orientable. 

  When the dimension of $V$ is $n=4k+1$, 
we have linearly independent $2k$ non-zero sections 
$(\omega_1)^k,\dots,(\omega_{2k})^k$ 
of the vector bundle $\bigwedge^{2k}\mathcal{D}^\ast\cong \mathcal{D}$ 
of rank~$2k+1$. 
 From the orientation of $\mathcal{D}\cong\bigwedge^{2k}\mathcal{D}^\ast$, 
we have another independent non-zero section. 
 Then, $\mathcal{D}$ is trivial. 
 When the dimension of $V$ is $n=4k+2$, we already have 
linearly independent $2k+1$ non-zero sections 
$(\omega_1)^k,\dots,(\omega_{2k+1})^k$ 
of the vector bundle $\bigwedge^{2k}\mathcal{D}^\ast\cong \mathcal{D}$ 
of rank~$2k+1$. 
 Then $\mathcal{D}$ is trivial. 

  This concludes the proof of Proposition~\ref{prop:ori}. 
\end{proof}

  From Proposition~\ref{prop:ori}, in the case when $k=1$ and the dimension 
of the manifold $V$ is $n=4k+1=5$, 
we obtain Corollary~\ref{corA} as follows. 
\begin{proof}[Proof of Corollary~\ref{corA}]
  For $k=1$, 
we obtain from Proposition~\ref{prop:ori} 
that a manifold $V$ of dimension~$5$ 
admits a maximally non-integrable orientable distribution of type~$(3,5)$ 
if and only if 
there exists a trivial subbundle $\mathcal{E}\subset TV$ of rank~$3$. 
 As we mentioned in Section~\ref{sec:distr}, 
a tangent distribution of type~$(3,5)$ is automatically 
maximally non-integrable. 
 Thus we have proved Corollary~\ref{corA}. 
\end{proof}

%
%
%

\bigskip 

\begin{flushleft}
Department of Mathematics, \\
Hokkaido University, \\
Sapporo, 060--0810, Japan. \\
\medskip
e-mail: j-adachi@math.sci.hokudai.ac.jp
\end{flushleft}
\end{document}